\documentclass[11pt,twoside,a4paper]{article}
\usepackage{amsmath,amsthm, amssymb, amsfonts,latexsym,amscd}
\usepackage[dvips]{graphicx}
\newtheorem{theorem}{Theorem}[section]
\newtheorem{corollary}[theorem]{Corollary}
\newtheorem{lemma}[theorem]{Lemma}
\newtheorem{proposition}[theorem]{Proposition}

\newtheorem{remark}[theorem]{Remark}
\newtheorem{example}[theorem]{Example}

\begin{document}

\title{\bf Vertex connectivity of the power graph of a finite cyclic group}
\author{Sriparna Chattopadhyay \and Kamal Lochan Patra \and Binod Kumar Sahoo}
\date{}
\maketitle

\begin{abstract}
Let $n=p_1^{n_1}p_2^{n_2}\ldots p_r^{n_r}$, where $r,n_1,\ldots, n_r$ are positive integers and $p_1,p_2,\ldots,p_r$ are distinct prime numbers with $p_1<p_2<\cdots <p_r$. For the cyclic group $C_n$ of order $n$, let $\mathcal{P}(C_n)$ be the power graph of $C_n$ and $\kappa(\mathcal{P}(C_n))$ be the vertex connectivity of $\mathcal{P}(C_n)$. It is known that $\kappa(\mathcal{P}(C_n))=p_1^{n_1} -1$ if $r=1$. For $r\geq 2$, we determine the exact value of $\kappa(\mathcal{P}(C_n))$ when $2\phi(p_1\ldots p_{r-1})\geq p_1\ldots p_{r-1}$, and give an upper bound for $\kappa(\mathcal{P}(C_n))$ when $2\phi(p_1\ldots p_{r-1}) < p_1\ldots p_{r-1}$, which is sharp for many values of $n$ but equality need not hold always. \\

\end{abstract}

\section{Introduction}

The notion of directed power graph of a group was introduced by Kelarev et al. in \cite{ker1}, which was further extended to semigroups in \cite{ker1-1, ker2}. Chakrabarty et al. defined the notion of undirected power graph of a semigroup, in particular, of a group in \cite{ivy}. Since then many researchers have investigated both directed and undirected power graphs of groups from different view points. More on these graphs can be found in the survey paper \cite{survey} and the references therein.

Let $G$ be a finite group. The {\it power graph} $\mathcal{P}(G)$ of $G$ is the simple undirected graph with vertex set $G$, in which two distinct vertices are adjacent if and only if one of them can be obtained as a power of the other. Since $G$ is finite, the identity element is adjacent to all other vertices and so $\mathcal{P}(G)$ is connected. The {\it vertex connectivity} of $\mathcal{P}(G)$, denoted by $\kappa(\mathcal{P}(G))$, is the minimum number of vertices which need to be removed from the vertex set $G$ so that the induced subgraph of $\mathcal{P}(G)$ on the remaining vertices is disconnected or has only one vertex. The latter case arises only when $\mathcal{P}(G)$ is a complete graph.

Note that $\kappa(\mathcal{P}(G))=|G|-1$ if and only if $\mathcal{P}(G)$ is a complete graph. Chakrabarty et al. proved in \cite[Theorem 2.12]{ivy} that $\mathcal{P}(G)$ is a complete graph if and only if $G$ is a cyclic group of prime power order. We have $\kappa(\mathcal{P}(G))=1$ if and only if $|G|=2$ or the subgraph $\mathcal{P}(G^{\ast})$ of $\mathcal{P}(G)$ is disconnected, where $G^{\ast}=G\setminus\{1\}$. This is equivalent to saying that $\kappa(\mathcal{P}(G))>1$ if and only if $|G|\geq 3$ and $\mathcal{P}(G^{\ast})$ is connected. Recently, some authors have studied the connectedness of $\mathcal{P}(G^{\ast})$, see \cite{doos} and \cite{mog} for related works in this direction.

Let $A(G)$ be the subset of $G$ consisting of those vertices which are adjacent to all other vertices of $\mathcal{P}(G)$. The identity element of $G$ is in $A(G)$. The groups $G$ for which $|A(G)|>1$ were described in \cite[Proposition 4]{cam}.

\begin{theorem}\cite{cam}\label{cam-1}
Let $G$ be a finite group for which $|A(G)|>1$. Then one of the following holds:
\begin{enumerate}
\item[(1)] $G$ is a cyclic group of prime power order and $A(G)=G$;
\item[(2)] $G$ is a cyclic group of non-prime-power order and $A(G)$ consists of the identity element and the generators of $G$;
\item[(3)] $G$ is a generalized quaternion $2$-group and $A(G)$ contains the identity element and the unique involution in $G$.
\end{enumerate}
\end{theorem}

Determining $\kappa(\mathcal{P}(G))$ for an arbitrary finite group $G$ is an interesting problem in the study of power graphs. When $\mathcal{P}(G)$ is not a complete graph, in order to find $\kappa(\mathcal{P}(G))$, first of all one needs to remove the vertices in $A(G)$ from $G$. This gives $\kappa(\mathcal{P}(G))\geq |A(G)|$. If $G$ is not a cyclic group, then the following upper bound for $\kappa(\mathcal{P}(G))$ was obtained in \cite[Theorem 10]{mir}.

\begin{theorem}\cite{mir}
Let $G$ be a non-cyclic finite group. Let $M(G)$ be the subset of $G$ consisting of elements $x$ for which $\langle x\rangle$ is a maximal cyclic subgroup of $G$. For $x\in M(G)$, define
$$r(x)=\underset{y\in M(G)\setminus\langle x\rangle}{\bigcup} (\langle x\rangle \cap \langle y\rangle).$$
Then $$\kappa(\mathcal{P}(G))\leq \min \{|r(x)|:x\in M(G)\}.$$
\end{theorem}

To our knowledge, $\kappa(\mathcal{P}(G))$ is not determined even when $G$ is an arbitrary finite cyclic group. In this paper, we consider $G=C_n$, a finite cyclic group of order $n$ and study the vertex connectivity $\kappa(\mathcal{P}(C_n))$ of $\mathcal{P}(C_n)$. We note that the graph $\mathcal{P}(C_n)$ has the maximum number of edges among all the power graphs of finite groups of order $n$. This property of $\mathcal{P}(C_n)$ was conjectured by Mirzargar et al. in \cite{mir} and proved by Curtin and Pourgholi in \cite{cur}. So it is natural to expect that $\kappa(\mathcal{P}(C_n))$ would be large in general.

If $n=1$ or if $n$ is divisible by only one prime number, then $\mathcal{P}\left(C_{n}\right)$ is a complete graph and so $\kappa\left(\mathcal{P}\left(C_{n}\right)\right)=n-1$. Therefore, for the rest of the paper, we assume that $n$ is divisible by at least two distinct prime numbers. We write the prime power decomposition of $n$ as
$$n=p_1^{n_1}p_2^{n_2}\ldots p_r^{n_r},$$
where $r\geq 2$, $n_1,\ldots, n_r$ are positive integers and $p_1,\ldots,p_r$ are distinct prime numbers with
$$p_1<p_2<\cdots <p_r.$$
Since $r\geq 2$, the set $A(C_n)$ consists of the identity element and the generators of $C_n$ (see Theorem \ref{cam-1}). So
\begin{equation}\label{eqn-0}
\kappa(\mathcal{P}(C_n))\geq |A(C_n)|=\phi(n) +1,
\end{equation}
where $\phi:\mathbb{N}\longrightarrow \mathbb{N}$ is the Euler's totient function. Note that the number of generators of $C_n$ is $\phi(n)$. If $n=p_1p_2$, then equality holds in (\ref{eqn-0}) by \cite[Theorem 3(ii)]{sri}. We shall show that the converse is also true (see Lemma \ref{iff}). When $n=p_1^{n_1}p_2^{n_2}$, it was proved in \cite[Theorem 2.7]{CP-2015} that
$$\kappa(\mathcal{P}(C_n))\leq \phi(n)+p_1^{n_1 -1}p_2^{n_2 -1}.$$
If $n=p_1p_2p_3$, then \cite[Theorem 2.9]{CP-2015} gives that
$$\kappa(\mathcal{P}(C_n))\leq \phi (n)+ p_1p_2 - \phi(p_1p_2).$$
Here we determine the exact value of $\kappa(\mathcal{P}(C_n))$ when $2\phi(p_1\ldots p_{r-1})\geq p_1\ldots p_{r-1}$ and give an upper bound for $\kappa(\mathcal{P}(C_n))$ when $2\phi(p_1\ldots p_{r-1}) < p_1\ldots p_{r-1}$.

\subsection{Main Results}

For a given subset $X$ of $C_n$, we define $\overline{X}=C_n\setminus X$ and denote by $\mathcal{P}(\overline{X})$ the induced subgraph of $\mathcal{P}(C_n)$ with vertex set $\overline{X}$. We prove the following in Sections \ref{upperbounds} and \ref{p1>=r -1}.

\begin{theorem}\label{main}
Let $n=p_1^{n_1}\ldots p_r^{n_r}$, where $r\geq 2$, $n_1,\ldots, n_r$ are positive integers and $p_1,\ldots,p_r$ are distinct prime numbers with
$p_1<p_2<\cdots <p_r$. Then the following hold:
\begin{enumerate}
\item[(i)]If $2\phi(p_1\ldots p_{r-1}) > p_1\ldots p_{r-1}$, then
$$\kappa(\mathcal{P}(C_n))=\phi(n) + p_1^{n_1 -1}\ldots p_{r-1}^{n_{r-1} -1} p_r^{n_r -1} \left[p_1p_2\ldots p_{r-1} - \phi(p_1p_2\ldots p_{r-1})\right].$$
Further, there is only one subset $X$ of $C_n$ with $|X|=\kappa(\mathcal{P}(C_n))$ such that $\mathcal{P}(\overline{X})$ is disconnected.
\item[(ii)] If $2\phi(p_1\ldots p_{r-1}) < p_1\ldots p_{r-1}$, then
$$\kappa(\mathcal{P}(C_n))\leq \phi(n) + p_1^{n_1 -1}\ldots p_{r-1}^{n_{r-1} -1} \left[p_1\ldots p_{r-1} + \phi(p_1\ldots p_{r-1})(p_r^{n_r -1}-2)\right].$$
\item[(iii)] If $2\phi(p_1\ldots p_{r-1}) = p_1\ldots p_{r-1}$, then $r=2$, $p_1=2$ (so that $n=2^{n_1} p_2^{n_2}$) and
$$\kappa(\mathcal{P}(C_n))= \phi(n) + 2^{n_1 -1} p_2^{n_2 -1}.$$
Moreover, there are $n_2$ subsets $X$ of $C_n$ with $|X|=\kappa(\mathcal{P}(C_n))$ for which $\mathcal{P}(\overline{X})$ is disconnected.
\end{enumerate}
\end{theorem}

Note that the values in the right hand side of the (in)equalities in Theorem \ref{main} are equal if $n_r=1$. As a consequence of Theorem \ref{main}(i) and (iii), we obtain the following when the smallest prime divisor of $n$ is at least the number of distinct prime divisors of $n$.

\begin{corollary}\label{coro}
If $p_1\geq r$, then
$$\kappa(\mathcal{P}(C_n))= \phi(n) + p_1^{n_1 -1}\ldots p_{r-1}^{n_{r-1} -1} p_r^{n_r -1} \left[p_1p_2\ldots p_{r-1} - \phi(p_1p_2\ldots p_{r-1})\right].$$
\end{corollary}

In order to show that the bound in Theorem \ref{main}(ii) is sharp for many values of $n$, we prove the following in Section \ref{last-sec} when $r= 3$.

\begin{theorem}\label{main-1}
Let $n=p_1^{n_1}p_2^{n_2}p_3^{n_3}$, where $n_1,n_2,n_3$ are positive integers and $p_1,p_2,p_3$ are distinct prime numbers with
$p_1<p_2<p_3$. If $2\phi(p_1p_2) < p_1p_2$, then $p_1 =2$ and
$$\kappa(\mathcal{P}(C_n))=\phi(n)+2^{n_{1}-1}p_{2}^{n_{2}-1}\left[(p_2 -1)p_{3}^{n_{3}-1}+2\right].$$
Further, there is only one subset $X$ of $C_n$ with $|X|=\kappa(\mathcal{P}(C_n))$ such that $\mathcal{P}(\overline{X})$ is disconnected.
\end{theorem}

We further note that, for certain values of $n$, equality may not hold in the bound obtained in Theorem \ref{main}(ii), see Example \ref{example} below. So it would be interesting to find $\kappa(\mathcal{P}(C_n))$ when $2\phi(p_1\ldots p_{r-1}) < p_1\ldots p_{r-1}$.

\section{Preliminaries}

Recall that $\phi$ is a multiplicative function, that is, $\phi(ab)=\phi(a)\phi(b)$ for any two positive integers $a,b$ which are relatively prime. So
$$\phi(n) = p_1^{n_1 -1}(p_1 -1) \ldots p_r^{n_r -1}(p_r -1)= p_1^{n_1 -1}\ldots p_r^{n_r -1}\phi(p_1 p_2\ldots p_r).$$

\begin{lemma}\label{simple-1}
For $1\leq i\leq r-1$, we have
$$\phi\left(\frac{n}{p_i}\right)\geq \phi\left(\frac{n}{p_r^{n_r}}\right)p_r^{n_r -1},$$
where the inequality is strict except when $r=2$, $p_1=2$, $p_2=3$ and $n_1\geq 2$.
\end{lemma}

\begin{proof}
Let $1\leq i\leq r-1$. We have $\phi(p_r)=p_r -1\geq p_{i}>p_{i} -1 =\phi(p_{i})$. If $n_i=1$, then
\begin{align*}
\phi\left(\frac{n}{p_i}\right) & = p_1^{n_1 -1}\ldots p_{i-1}^{n_{i-1}-1}p_{i+1}^{n_{i+1}-1}\ldots p_r^{n_r -1} \phi(p_1 \ldots p_{i-1} p_{i+1}\ldots p_{r})\\
 & > p_1^{n_1 -1}\ldots p_{i-1}^{n_{i-1}-1}p_{i+1}^{n_{i+1}-1}\ldots p_r^{n_r -1} \phi(p_1 \ldots p_{i} \ldots p_{r-1}) = \phi\left(\frac{n}{p_r^{n_r}}\right)p_r^{n_r -1}.
\end{align*}
If $n_i\geq 2$, then
\begin{align*}
\phi\left(\frac{n}{p_i}\right) & = p_1^{n_1 -1}\ldots p_{i-1}^{n_{i-1}-1}p_i^{n_i -2}p_{i+1}^{n_{i+1}-1}\ldots p_r^{n_r -1} \phi(p_1 p_2\ldots p_{r})\\
 & \geq p_1^{n_1 -1}\ldots p_i^{n_i -1}\ldots p_{r-1}^{n_{r-1} -1} p_r^{n_r -1} \phi(p_1 p_2 \ldots p_{r-1})=\phi\left(\frac{n}{p_r^{n_r}}\right)p_r^{n_r -1}.
\end{align*}
The last part of the lemma can be seen easily. Note that $\phi(p_r)> p_{i}$ for $r\geq 3$.
\end{proof}

\begin{lemma}\label{simple-4}
We have $\phi\left(p_{1}p_2\ldots p_{t}\right)- p_1 p_2 \ldots p_{t} +\underset{k=1}{\overset{t}{\sum}} \frac{p_1 p_2 \ldots p_{t}}{p_k} \geq 0$ for all $t\geq 1$. Further, equality holds if and only if $t=1$.
\end{lemma}

\begin{proof}
Clearly, equality holds if $t=1$.
So let $t\geq 2$. We shall prove by induction on $t$ that strict inequality holds.
Since $\phi\left(p_{1}p_2\right)- p_1 p_2+ p_1 +p_2 =1 > 0$, the statement is true for $t=2$. Assume that $t\geq 3$.
Then
\begin{eqnarray*}
& & \phi\left(p_{1}p_2\ldots p_{t}\right)- p_1 p_2 \ldots p_{t} + \underset{k=1}{\overset{t}{\sum}} \frac{p_1 p_2 \ldots p_{t}}{p_k}\\
& =& \phi\left(p_{1}\ldots p_{t-1}\right)(p_{t}-1)- p_1 \ldots p_{t-1}p_{t} + \left[ \underset{k=1}{\overset{t-1}{\sum}} \frac{p_1 p_2 \ldots p_{t-1}}{p_k}\right]\times p_t + p_1 p_2 \ldots p_{t-1}\\
&= & p_t\times \left[\phi\left(p_{1}\ldots p_{t-1}\right)- p_1 \ldots p_{t-1} +\underset{k=1}{\overset{t-1}{\sum}} \frac{p_1 \ldots p_{t-1}}{p_k}\right] + p_1\ldots p_{t-1} - \phi\left(p_{1}\ldots p_{t-1}\right)\\
& > & \phi\left(p_{1}\ldots p_{t-1}\right)- p_1 \ldots p_{t-1} +\underset{k=1}{\overset{t-1}{\sum}} \frac{p_1 \ldots p_{t-1}}{p_k} > 0.
\end{eqnarray*}
In the above, the last inequality follows using the induction hypothesis.
\end{proof}

For $x\in C_n$, we denote by $o(x)$ the order of $x$. If two elements $x,y$ are adjacent in $\mathcal{P}(C_n)$, then $o(x)\mid o(y)$ or $o(y)\mid o(x)$ according as $x\in\langle y\rangle$ or $y\in\langle x\rangle$. The converse statement is also true (which does not hold for an arbitrary finite group), which follows from the property that $C_n$ has a unique subgroup of order $d$ for every positive divisor $d$ of $n$. We shall use this fact frequently without mention. For $x\in C_n$, let $N(x)$ be the neighborhood of $x$ in $\mathcal{P}(C_n)$, that is, the set of all elements of $C_n$ which are adjacent to $x$. If $o(x)=o(y)$ for $x,y\in C_n$, then it is clear that $N(x)\cup \{x\}=N(y)\cup \{y\}$, also see \cite[Lemma 3]{dooser}.

Let $X$ be a subset of $C_n$. For two disjoint nonempty subsets $A$ and $B$ of $\overline{X}$, we say that $A\cup B$ is a {\it separation} of $\mathcal{P}(\overline{X})$ if $\overline{X} =A\cup B$ and there is no edge containing vertices from both $A$ and $B$. Thus $\mathcal{P}(\overline{X})$ is disconnected if and only if there exists a separation of it. For a positive divisor $d$ of $n$, we define the following two sets:
\begin{enumerate}
\item[] $E_d$ = the set of all elements of $C_n$ whose order is $d$,
\item[] $S_d$ = the set of all elements of $C_n$ whose order divides $d$.
\end{enumerate}
Then $S_d$ is a cyclic subgroup of $C_n$ of order $d$ and $E_d$ is precisely the set of generators of $S_d$. So
$\left\vert S_d\right\vert =d$ and $\left\vert E_d\right\vert =\phi(d).$ The following result is very useful throughout the paper.

\begin{lemma}\label{simple-3}
Let $X$ be a subset of $C_n$ of minimum possible size with the property that $\mathcal{P}(\overline{X})$ is disconnected. Then either $E_d\subseteq X$ or $E_d\cap X$ is empty for each divisor $d$ of $n$.
\end{lemma}

\begin{proof}
Suppose that $E_d\cap X\neq \emptyset$. We show that $E_d\subseteq X$. Fix a separation $A\cup B$ of $\mathcal{P}(\overline{X})$. Let $x\in E_d\cap X$. The minimality of $|X|$ implies that the subgraph $\mathcal{P}(\overline{X\setminus\{x\}})$ of $\mathcal{P}(C_n)$ is connected. So there exist $a\in A$ and $b\in B$ such that $x$ is adjacent to both $a$ and $b$.

Suppose that there exists an element $y\in E_d$ which is not in $X$. Then $y\neq x$. Since $o(x)=o(y)=d$, we have $N(x)\cup \{x\}=N(y)\cup \{y\}$. If $y\in A$ (respectively, $y\in B$), then the fact that $x$ is adjacent to $b$ (respectively, to $a$) implies $y$ is adjacent to $b$ (respectively, to $a$). This contradicts that $A\cup B$ is a separation of $\mathcal{P}(\overline{X})$.
\end{proof}

\begin{remark}
Under the hypothesis of Lemma \ref{simple-3}, it follows that there are three possibilities for the set $E_d$, where $d$ is a divisor of $n$: either $E_d\subseteq X$, $E_d\subseteq A$ or $E_d\subseteq B$, where $A\cup B$ is any separation of $\mathcal{P}(\overline{X})$.
\end{remark}

We complete this section with the following lemma. For a given subset $\{i_1,i_2,\ldots, i_k\}$ of $\{1,2,\ldots, r\}$, we define the integer $m_{i_1, i_2,\ldots, i_k}$ by
$$m_{i_1, i_2,\ldots ,i_k}=\frac{n}{p_{i_1}p_{i_2}\ldots p_{i_k}}.$$

\begin{lemma}\label{iff}
$\kappa (\mathcal{P}(C_{n}))= \phi(n)+1$ if and only if $r=2$ and $n=p_{1}p_{2}$.
\end{lemma}
\begin{proof}
If $n=p_{1}p_{2}$, then $\kappa (\mathcal{P}(C_{n}))= \phi(n)+1$ by \cite[Theorem 3(ii)]{sri}. For the converse part, let $X=E_{1}\cup E_n$. It is enough to show that $\mathcal{P}(\overline{X})$ is connected whenever $r\geq 3$, or $r=2$ and one of $n_{1}, n_{2}$ is at least $2$.

Let $x$ and $y$ be two distinct elements of $\overline{X}$. Then $x\in S_{m_j}$ and $y\in S_{m_k}$ for some $j, k\in\{1,2,\ldots, r\}$. So $x$ (respectively, $y$) is adjacent with the elements of $E_{m_j}\setminus \{x\}$ (respectively, $E_{m_k}\setminus\{y\}$). If $j=k$, then $x$ and $y$ are connected through the elements of $E_{m_j}$. Assume that $j\neq k$. Since $r\geq 3$, or  $r=2$ and one of $n_{1}, n_{2}$ is at least $2$, the set $E_{m_{j,k}}$ is non-empty. Then the elements of both $E_{m_j}$ and $E_{m_k}$ are adjacent with the elements of $E_{m_{j,k}}$. It follows that $x$ and $y$ are connected by a path.
\end{proof}

\section{Upper bounds and Proof of Theorem \ref{main}(ii) and (iii)}\label{upperbounds}

We shall prove the bounds in Theorem \ref{main} by identifying suitable subsets $X$ of $C_n$ of required size such that $\mathcal{P}(\overline{X})$ is disconnected. Let $0\leq k\leq n_r -1$. Define the following integers:
$$\alpha_{k}=p_{1}^{n_{1}}p_{2}^{n_{2}}\ldots p_{r-1}^{n_{r-1}}p_{r}^{k}.$$
So $\alpha_k=\alpha_0 p_r^k$. For given $k$ and subset $\{i_1,i_2,\ldots, i_l\}$ of $\{1,2,\ldots, r-1\}$, define the integer $\beta_{k,i_1, i_2,\ldots, i_l}$ by
$$\beta_{k,i_1, i_2,\ldots, i_l}  = \frac{\alpha_k}{p_{i_1} \ldots p_{i_l}}=\frac{p_1^{n_1}\ldots p_{r-1}^{n_{r-1}}p_{r}^{k}}{p_{i_1} \ldots p_{i_l}}.$$
Set
$$Z(r,k)=E_{\alpha_{k+1}}\cup \ldots\cup E_{\alpha_{n_r -1}}\cup E_{n}\cup S_{\beta_{k,1}}\cup \ldots \cup S_{\beta_{k,r-1}}.$$

\begin{proposition}\label{Z-disconn}
The subgraph $\mathcal{P}(\overline{Z(r,k)})$ of $\mathcal{P}(C_n)$ is disconnected.
\end{proposition}

\begin{proof}
For $x\in \overline{Z(r,k)}$, observe that the order $o(x)$ of $x$ is one of the following two types:
\begin{enumerate}
\item[(I)] $o(x)=p_1^{n_1}p_2^{n_2} \ldots p_{r-1}^{n_{r-1}}p_{r}^{s}$, where $0\leq s\leq k$;
\item[(II)] $o(x)=p_1^{l_1}\ldots p_{r-1}^{l_{r-1}}p_r^{t}$, where $k+1\leq t\leq n_r$, $0\leq l_i\leq n_i$ for $i\in\{1,\ldots, r-1\}$ and $(l_1,\ldots,l_{r-1})\neq (n_1,\ldots, n_{r-1})$.
 \end{enumerate}
Let $A$ (respectively, $B$) be the subset of $\overline{Z(r,k)}$ consisting of all the elements whose order is of type (I) (respectively, of type (II)).
Then $A,B$ are nonempty sets and $\overline{Z(r,k)} = A\cup B$ is a disjoint union. Since $t>s$, no element of $B$ can be obtained as an integral power of any element of $A$. Again, since $(l_1,\ldots,l_{r-1})\neq (n_1,\ldots, n_{r-1})$, no element of $A$ can be obtained as an integral power of any element of $B$. It follows that $A\cup B$ is a separation of $\mathcal{P}(\overline{Z(r,k)})$.
\end{proof}

\begin{proposition}\label{size-Z}
The number of elements in $Z(r,k)$ is given by:
$$|Z(r,k)| =\phi(n) + \beta_{0,1,\ldots, r-1} \left[p_r^{n_r -1}\phi(p_1\ldots p_{r-1})+p_r^k\left[p_1\ldots p_{r-1} -2 \phi(p_1\ldots p_{r-1})\right]\right].$$
\end{proposition}

\begin{proof}
The sets $E_{\alpha_{k+1}}, \ldots, E_{\alpha_{n_r -1}}, E_n$ are pairwise disjoint and each of them is disjoint from $\underset{i=1}{\overset{r-1}{\bigcup}} S_{\beta_{k,i}}$. So
$$|Z(r,k)|=\left\vert E_{n}\right\vert + \underset{i=k+1}{\overset{n_r-1}{\bigcup}} \left\vert E_{\alpha_{i}}\right\vert+\left\vert \underset{i=1}{\overset{r-1}{\bigcup}} S_{\beta_{k,i}}\right\vert.$$
We have
\begin{align*}
\underset{i=k+1}{\overset{n_r-1}{\bigcup}} \left\vert E_{\alpha_{i}}\right\vert & = \phi(\alpha_0)\left[\phi\left(p_r^{k+1}\right)+\phi\left(p_r^{k+2}\right)+\ldots + \phi\left(p_r^{n_r -1}\right)\right] \\
&  =  p_1^{n_1 -1}p_2^{n_2 -1} \ldots p_{r-1}^{n_{r-1} -1}\phi(p_1p_2\ldots p_{r-1})(p_r^{n_r -1} -p_r^{k}).
\end{align*}
and
\begin{align*}
\left\vert \underset{i=1}{\overset{r-1}{\bigcup}} S_{\beta_{k,i}}\right\vert & = \underset{i}{\sum} \left\vert S_{\beta_{k,i}}\right\vert  - \underset{i<j}{\sum} \left\vert S_{\beta_{k,i}}\cap S_{\beta_{k,j}}\right\vert
+ \ldots + (-1)^{r-2} \left\vert S_{\beta_{k,1}}\cap \ldots \cap S_{\beta_{k,r-1}}\right\vert \\
 & = \underset{i}{\sum} \beta_{k,i} - \underset{i<j}{\sum} \beta_{k,i,j} + \ldots + (-1)^{r-2}\beta_{k,1,2,\ldots, r-1}\\
 &  =  p_1^{n_1 -1}\ldots p_{r-1}^{n_{r-1} -1}p_r^k\left[\underset{i=1}{\overset{r-1}{\sum}}\frac{p_{1}\ldots p_{r-1}}{p_{i}} - \underset{i<j}{\sum}\frac{p_{1}\ldots p_{r-1}}{p_{i}p_{j}}+\ldots +(-1)^{r-2}\right]\\
 & = p_1^{n_1 -1} \ldots p_{r-1}^{n_{r-1} -1}p_r^k\left[p_1p_2\ldots p_{r-1}-\phi(p_1p_2\ldots p_{r-1})\right].
\end{align*}
Now the lemma follows from the above, as $\beta_{0,1,\ldots, r-1}=p_1^{n_1 -1} \ldots p_{r-1}^{n_{r-1} -1}$.
\end{proof}

In the next section, we shall have occasions to calculate the cardinality of the union of certain subgroups, which will be similar to that of calculating $\left\vert \underset{i=1}{\overset{r-1}{\bigcup}} S_{\beta_{k,i}}\right\vert$ as in the above. As a consequence of Propositions \ref{Z-disconn} and \ref{size-Z}, we have the following.

\begin{corollary}
The vertex connectivity $\kappa(\mathcal{P}(C_n))$ of $\mathcal{P}(C_n))$ satisfies the following:
$$\kappa(\mathcal{P}(C_n))\leq \phi(n) + p_1^{n_1 -1}\ldots p_{r-1}^{n_{r-1}-1} \left[p_r^{n_r -1}\phi(p_1\ldots p_{r-1})+p_r^k\left[p_1\ldots p_{r-1} -2 \phi(p_1\ldots p_{r-1})\right]\right].$$
\end{corollary}

\begin{proof}[{\bf Proof of Theorem \ref{main}(ii)}]
If $2\phi(p_1\ldots p_{r-1}) < p_1\ldots p_{r-1}$, then the minimum of $|Z(r,k)|$ occurs when $k=0$. This gives
\begin{equation}\label{bound-2}
\kappa(\mathcal{P}(C_n))\leq \phi(n) + p_1^{n_1 -1}\ldots p_{r-1}^{n_{r-1} -1} \left[p_1\ldots p_{r-1} + \phi(p_1\ldots p_{r-1})(p_r^{n_r -1}-2)\right],
\end{equation}
thus proving Theorem \ref{main}(ii).
\end{proof}

The following example shows that equality may not hold in (\ref{bound-2}).

\begin{example}\label{example}
Consider $n=2310=2\times 3\times 5\times 7\times 11$. The value in the right hand side of (\ref{bound-2}) is $\phi(n) +162$. Now set
$$X=E_n\cup E_{210}\cup E_{330}\cup S_{6}\cup S_{10}\cup S_{15}.$$
Taking $A=E_{30}$ and $B=C_n\setminus(X\cup A)$, it can be seen that $A\cup B$ is a separation of $\mathcal{P}(\overline{X})$ and so $\mathcal{P}(\overline{X})$ is disconnected. Calculating $|X|$, we get
$$|X|= \phi(n)+150 < \phi(n)+162.$$
So equality may not hold in (\ref{bound-2}).
\end{example}

If $2\phi(p_1\ldots p_{r-1}) > p_1\ldots p_{r-1}$, then the minimum of $|Z(r,k)|$ occurs when $k=n_r -1$ and this gives
\begin{equation}\label{bound-1}
\kappa(\mathcal{P}(C_n))\leq \phi(n) + p_1^{n_1 -1}\ldots p_{r-1}^{n_{r-1} -1} p_r^{n_r -1} \left[p_1p_2\ldots p_{r-1} - \phi(p_1p_2\ldots p_{r-1})\right].
\end{equation}
In the next section, we shall show that equality holds in (\ref{bound-1}) which will prove Theorem \ref{main}(i). Observe that the bounds (\ref{bound-2}) and (\ref{bound-1}) coincide if $n_r =1$, or if $r=2$ and $p_1=2$.

\begin{proof}[{\bf Proof of Theorem \ref{main}(iii)}]
Note that $2 \phi(p_1\ldots p_{r-1})=p_1\ldots p_{r-1}$ if and only if $r=2$ and $p_1=2$. So $n=2^{n_1}p_2^{n_2}$ in this case. For $r=2$ and $0\leq k\leq n_2 -1$, we have
$$Z(2,k)=E_{2^{n_1}p_2^{k+1}}\cup E_{2^{n_1}p_2^{k+2}}\cup \ldots \cup E_{2^{n_1}p_2^{n_2 -1}} \cup E_{n}\cup S_{2^{n_1 -1}p_2^{k}}$$
and that $|Z(2,k)|=\phi(n) + 2^{n_1 -1} p_2^{n_2 -1}$ is independent of $k$.
Thus
$$\kappa\left(\mathcal{P}\left(C_{n}\right)\right)\leq |Z(2,k)|=\phi\left(n\right) + 2^{n_1 -1} p_2^{n_2 -1}.$$

Now, let $X$ be a subset of $C_n$ of minimum possible size such that $\mathcal{P}(\overline{X})$ is disconnected. In order to prove Theorem \ref{main}(iii), it is enough to show that $X=Z(2,t)$ for some $0\leq t\leq n_2 -1$.
Write $T= X\setminus (E_1\cup E_n)$. Since $X$ contains $E_1\cup E_n$ and $|X|\leq \phi(n) + 2^{n_1 -1}p_2^{n_2 -1}$, we get
$$|T|\leq 2^{n_1 -1}p_2^{n_2 -1}-1.$$
We claim that the set $E_{2^{n_1 -1}p_2^{n_2}}$ is disjoint from $X$.
Otherwise, $E_{2^{n_1 -1}p_2^{n_2}}\subseteq T$ by Lemma \ref{simple-3} and so $\left\vert E_{2^{n_1 -1}p_2^{n_2}}\right\vert\leq |T|$. On the other hand, using Lemma \ref{simple-1}, we get
$$\left\vert E_{2^{n_1 -1}p_2^{n_2}}\right\vert = \phi(2^{n_1 -1}p_2^{n_2})\geq \phi(2^{n_1})p_2^{n_2 -1}>  2^{n_1 -1}p_2^{n_2 -1}-1\geq |T|,$$
a contradiction.

Let $A\cup B$ be a separation of $\mathcal{P}(\overline {X})$. We may assume that $E_{2^{n_1 -1}p_2^{n_2}}$ is contained in $B$. Then, for $x\in A$, we must have
$$o(x)=2^{n_1}p_2^{j}$$
for some $j$ with $0\leq j\leq n_2 -1$. Let $t\in \{0,1,\ldots, n_2 -1\}$ be the largest integer for which $A$ has an element of order $2^{n_1}p_2^{t}$. Then, using Lemma \ref{simple-3} together with the fact that $A\cup B$ is a separation of $\mathcal{P}(\overline {X})$, the following hold:
\begin{enumerate}
\item[(i)] $E_{2^{n_1}p_2^{t}}\subseteq A$,
\item[(ii)] $E_{2^{n_1}p_2^{t+1}}, E_{2^{n_1}p_2^{t+2}},\ldots, E_{2^{n_1}p_2^{n_2 -1}}$ are contained in $T$,
\item[(iii)] the subgroup $S_{2^{n_1 -1}p_2^{t}}$ is contained in $X$.
\end{enumerate}
Thus $X$ contains $Z(2,t)$. Since $|X|\leq \phi(n) + 2^{n_1 -1}p_2^{n_2 -1}=|Z(2,t)|$, it follows that $X=Z(2,t)$. This completes the proof.
\end{proof}

\section{Proof of Theorem \ref{main}(i) and Corollary \ref{coro}}\label{p1>=r -1}

Assume, throughout this section, that $2 \phi(p_1\ldots p_{r-1})> p_1\ldots p_{r-1}$ (and so $p_1\geq 3$). Let $X$ be a subset of $C_n$ of minimum possible size with the property that $\mathcal{P}(\overline{X})$ is disconnected. By (\ref{bound-1}),
$$|X|\leq \phi(n) + p_1^{n_1 -1}p_2^{n_2 -1} \ldots p_r^{n_r -1} \left[p_1p_2\ldots p_{r-1} - \phi(p_1p_2\ldots p_{r-1})\right].$$
Set $T=X\setminus E_n$. Since $E_n\subseteq X$, we have
\begin{equation}\label{eqn-1}
|T| \leq p_1^{n_1 -1}p_2^{n_2 -1} \ldots p_r^{n_r -1} \left[p_1p_2\ldots p_{r-1} - \phi(p_1p_2\ldots p_{r-1})\right].
\end{equation}

\begin{proposition}\label{disjoint}
Each of the sets $E_{m_i}$, $1\leq i\leq r-1$, is disjoint from $X$.
\end{proposition}

\begin{proof}
Suppose that $E_{m_i}\cap X\neq \emptyset$. Then $E_{m_i}\subseteq X$ by Lemma \ref{simple-3}, in fact, $E_{m_i}\subseteq T$. So $|E_{m_i}|\leq |T|$. Since $1\leq i\leq r-1$,
$|E_{m_i}| = \phi\left(\frac{n}{p_i}\right) > \phi\left(\frac{n}{p_r^{n_r}}\right)p_r^{n_r -1}$ by Lemma \ref{simple-1} and so
\begin{equation}\label{eqn-2}
|E_{m_i}| > p_1^{n_1 -1} p_2^{n_2 -1}\ldots p_r^{n_r -1}\phi(p_1 p_2\ldots p_{r-1}).
\end{equation}
Then the inequalities (\ref{eqn-1}) and (\ref{eqn-2}) together imply
$$|E_{m_i}|-|T| > p_1^{n_1 -1}p_2^{n_2 -1} \ldots p_r^{n_r -1} [2 \phi(p_1p_2\ldots p_{r-1}) - p_1p_2\ldots p_{r-1}].$$
Since $2\phi(p_1p_2\ldots p_{r-1})> p_1p_2\ldots p_{r-1}$, it follows that $|E_{m_i}|>|T|$, a contradiction.
\end{proof}

We shall prove later that the set $E_{m_r}$ is also disjoint from $X$. However, the argument used in the proof of Proposition \ref{disjoint} can not be applied (when $n_r\geq 2$) to prove this statement.

Fix a separation $A\cup B$ of $\mathcal{P}(\overline{X})$. Proposition \ref{disjoint} implies that each $E_{m_i}$, $1\leq i\leq r-1$, is contained either in $A$ or in $B$.

\begin{proposition}\label{E-m-r}
Suppose that $r\geq 3$. If $E_{m_i}\subseteq A$ and $E_{m_j}\subseteq B$ for some $1\leq i\neq j\leq r-1$, then $E_{m_r}$ is disjoint from $X$.
\end{proposition}

\begin{proof}
Since $A\cup B$ is a separation of $\mathcal{P}(\overline{X})$, the subgroup $S_{m_{i,j}}$ of $C_n$ is contained in $T$. Suppose that $E_{m_r}\cap X\neq \emptyset$. Then $E_{m_r}\subseteq T$ by Lemma \ref{simple-3}. So $|S_{m_{i,j}}| + |E_{m_r}|=|S_{m_{i,j}}\cup E_{m_r}|\leq |T|$. If $n_r\geq 2$, then
\begin{align*}
|S_{m_{i,j}}| + |E_{m_r}| & = \frac{n}{p_i p_j} + p_1^{n_1 -1}\ldots p_{r-1}^{n_{r-1} -1}p_r^{n_r -2}\phi(p_1 p_2\ldots p_r)\\
 & =  p_1^{n_1 -1}\ldots p_{r-1}^{n_{r-1} -1}p_r^{n_r -2}\left(p_r^2\underset{\underset{k\neq i,j}{k=1}}{\overset{r-1}{\prod}} p_k +\phi(p_1\ldots p_{r})\right)\\
 &>  p_1^{n_1 -1}\ldots p_{r-1}^{n_{r-1} -1}p_r^{n_r -2}\left(\underset{k=1}{\overset{r-1}{\prod}} p_k + \phi(p_1\ldots p_{r})\right),
\end{align*}
and so
\begin{align*}
\frac{|S_{m_{i,j}}|+|E_{m_r}|-|T|}{p_1^{n_1 -1}\ldots p_{r-1}^{n_{r-1} -1}p_r^{n_r -2}} & >\underset{k=1}{\overset{r-1}{\prod}} p_k + \phi(p_1\ldots p_{r})-p_r\left[\underset{k=1}{\overset{r-1}{\prod}} p_k - \phi(p_1\ldots p_{r-1})\right]\\
 &= \underset{k=1}{\overset{r-1}{\prod}} p_k +\phi(p_1\ldots p_{r-1})(p_r-1)-p_r\left[\underset{k=1}{\overset{r-1}{\prod}} p_k - \phi(p_1\ldots p_{r-1})\right]\\
 &= \underset{k=1}{\overset{r-1}{\prod}} p_k - \phi(p_1\ldots p_{r-1}) +p_r\left[2\phi(p_1\ldots p_{r-1})- \underset{k=1}{\overset{r-1}{\prod}} p_k\right] > 0.
\end{align*}
If $n_r=1$, then $|S_{m_{i,j}}| + |E_{m_r}|  = \frac{n}{p_i p_j} + p_1^{n_1 -1}\ldots p_{r-1}^{n_{r-1} -1}\phi(p_1 p_2\ldots p_{r-1})$ and so
\begin{align*}
\frac{|S_{m_{i,j}}|+|E_{m_r}|-|T|}{p_1^{n_1 -1}\ldots p_{r-1}^{n_{r-1} -1}} & \geq \underset{\underset{k\neq i,j}{k=1}}{\overset{r}{\prod}} p_k + 2\phi(p_1\ldots p_{r-1})- p_1 p_2\ldots p_{r-1}>0.
\end{align*}
In both cases, it follows that $|S_{m_{i,j}}|+|E_{m_r}|>|T|$, a contradiction.
\end{proof}

\begin{proposition}\label{all}
All the sets $E_{m_i}$, $1\leq i\leq r-1$, are contained either in $A$ or in $B$.
\end{proposition}

\begin{proof}
Clearly, this holds for $r=2$. Assume that $r\geq 3$. Suppose that some of the sets $E_{m_i}$, $1\leq i\leq r-1$, are contained in $A$ and some are in $B$. Then $E_{m_r}$ is disjoint from $X$ by Proposition \ref{E-m-r}. Without loss, we may assume that $E_{m_r}$ is contained in $A$. Let
$$a=\left\vert \left\{E_{m_i}: 1\leq i\leq r-1, E_{m_i}\subseteq A\right\}\right\vert$$
Then $1\leq a\leq r-2$ by our assumption. We shall get a contradiction by showing that $a\notin\{1,2,\ldots, r-2\}$.\\

{\bf Claim-1:} $a\neq r-2$.
Suppose that $a=r-2$. Let $E_{m_{i_1}},E_{m_{i_2}},\ldots,E_{m_{i_a}}$ be the sets contained in $A$ and $E_{m_{i_{a+1}}}$ be contained in $B$, where $\{i_1,\ldots,i_a,i_{a+1}\}=\{1,2,\ldots,r-1\}$. Since $E_{m_r}$ is contained in $A$, the subgroups
$$S_{m_{r,i_{a+1}}}, S_{m_{i_1,i_{a+1}}}, S_{m_{i_2,i_{a+1}}},\ldots , S_{m_{i_a,i_{a+1}}}$$
are contained in $T$. This follows as $A\cup B$ is a separation of $\mathcal{P}(\overline{X})$. Let $Q$ be the union of these $r-1$ subgroups.
Then $|Q|\leq |T|$. We shall get a contradiction by showing that $|Q|>|T|$.

The subscript $i_{a+1}$ is common to all the above $r-1$ subgroups. Applying a similar calculation as in the proof of Proposition \ref{size-Z}, we get
\begin{align*}
\left\vert Q\right\vert &
= p_{i_1}^{n_{i_1}-1}\ldots p_{i_a}^{n_{i_a}-1}p_{i_{a+1}}^{n_{i_{a+1}}-1}p_{r}^{n_{r}-1}\left[p_{i_{1}}\ldots p_{i_{a}}p_r-\phi(p_{i_{1}}\ldots p_{i_{a}}p_r)\right]\\
& =  p_{1}^{n_{1}-1}p_{2}^{n_{2}-1}\ldots p_r^{n_r -1}\left[p_{i_{1}}\ldots p_{i_{a}}p_r-\phi(p_{i_{1}}\ldots p_{i_{a}}p_r)\right].
\end{align*}
Then
\begin{align*}
\frac{|Q|-|T|}{p_{1}^{n_{1}-1}\ldots p_r^{n_r -1}} & \geq p_{i_{1}}\ldots p_{i_{a}}p_r-\phi(p_{i_{1}}\ldots p_{i_{a}}p_r)- p_1 p_2\ldots p_{r-1} + \phi(p_1p_2\ldots p_{r-1})\\
 & = p_{i_{1}}\ldots p_{i_{a}}(p_r - p_{i_{a+1}}) + \phi(p_{i_{1}}\ldots p_{i_{a}})\left[\phi(p_{i_{a+1}})- \phi(p_r)\right]\\
 & = (p_r - p_{i_{a+1}})(p_{i_{1}}\ldots p_{i_{a}} - \phi(p_{i_{1}}\ldots p_{i_{a}})) > 0.
\end{align*}
The last inequality holds as $1\leq i_{a+1}\leq r-1$. So $|Q|>|T|$, a contradiction. This proves Claim-1.\\

{\bf Claim-2:} $a\notin\{1,2,\ldots, r-3\}$.
Suppose that $1\leq a\leq r-3$ (we must have $r\geq 4$ as $a\geq 1$). Set $b=r-1-a$. Then $b\geq 2$. Let $E_{m_{i_1}},\ldots,E_{m_{i_a}}$ be the sets contained in $A$ and $E_{m_{i_{a+1}}},E_{m_{i_{a+2}}},\ldots,E_{m_{i_{a+b}}}$ be contained in $B$, where
$$\{i_1,\ldots,i_a,i_{a+1},\ldots, i_{a+b}\}=\{1,2,\ldots,r-1\}.$$
So the following subgroups
\begin{align*}
& S_{m_{r,i_{a+1}}}, S_{m_{r,i_{a+2}}},\ldots, S_{m_{r,i_{a+b}}}\\
& S_{m_{i_1,i_{a+1}}}, S_{m_{i_1,i_{a+2}}},\ldots, S_{m_{i_1,i_{a+b}}}\\
& \vdots \\
& S_{m_{i_a,i_{a+1}}}, S_{m_{i_a,i_{a+2}}},\ldots, S_{m_{i_a,i_{a+b}}}
\end{align*}
are contained in $T$. Let $1\leq s\leq a$ and $1\leq t\leq b$. Note that $E_{m_{i_s,i_{a+t}}}$ is the set of generators of the subgroup $S_{m_{i_s,i_{a+t}}}$ and so is contained in $T$. Define the set
$$R=\left(\underset{j=1}{\overset{b}{\bigcup}} S_{m_{r,i_{a+j}}}\right) \bigcup \left( \underset{t=1}{\overset{b}{\bigcup}}\left( \underset{s=1}{\overset{a}{\bigcup}} E_{m_{i_s,i_{a+t}}}\right)\right).$$
Since $R$ is contained in $T$, we have $|R|\leq |T|$. We shall get a contradiction by showing that $|R|>|T|$.

The subscript $r$ is common to all the $b$ subgroups contained in $R$. Applying a similar calculation as in the proof of Proposition \ref{size-Z}, we have
\begin{align*}
\left\vert \underset{j=1}{\overset{b}{\bigcup}} S_{m_{r,i_{a+j}}}\right\vert &
= p_{i_1}^{n_{i_1}}\ldots p_{i_a}^{n_{i_a}}p_{i_{a+1}}^{n_{i_{a+1}}-1}\ldots p_{i_{a+b}}^{n_{i_{a+b}}-1}p_r^{n_r -1}\left[p_{i_{a+1}}\ldots p_{i_{a+b}}-\phi(p_{i_{a+1}}\ldots p_{i_{a+b}})\right]\\
& =  p_{i_1}^{n_{i_1}-1}\ldots p_{i_{a+b}}^{n_{i_{a+b}}-1}p_r^{n_r -1}\left[p_{i_1}\ldots p_{i_a}p_{i_{a+1}}\ldots p_{i_{a+b}}- p_{i_1}\ldots p_{i_a}\phi(p_{i_{a+1}}\ldots p_{i_{a+b}})\right]\\
& =  p_{1}^{n_{1}-1}p_{2}^{n_{2}-1}\ldots p_r^{n_r -1}\left[p_{1}p_2\ldots p_{r-1}- p_{i_1}\ldots p_{i_a}\phi(p_{i_{a+1}}\ldots p_{i_{a+b}})\right].
\end{align*}
We next calculate a lower bound for $\left\vert E_{m_{i_s,i_{a+t}}}\right\vert$.
Applying Lemma \ref{simple-1},
\begin{equation}\label{eqn-3}
\phi\left(\frac{p_{i_{a+1}}^{n_{i_{a+1}}}\ldots p_{i_{a+b}}^{n_{i_{a+b}}}p_r^{n_r}}{p_{i_{a+t}}}\right)  > p_{i_{a+1}}^{n_{i_{a+1}-1}}\ldots p_{i_{a+b}}^{n_{i_{a+b}-1}}p_r^{n_r -1}\times \phi\left(p_{i_{a+1}}\ldots p_{i_{a+b}}\right).
\end{equation}
It can easily be seen (irrespective of $n_{i_{s}}=1$ or $n_{i_{s}}\geq 2$) that
\begin{equation}\label{eqn-4}
\phi\left(\frac{p_{i_1}^{n_{i_1}}\ldots p_{i_{a}}^{n_{i_{a}}}}{p_{i_s}}\right)  \geq p_{i_1}^{n_{i_1}-1}\ldots p_{i_{a}}^{n_{i_{a}}-1} \times \frac{\phi\left(p_{i_1}p_{i_2}\ldots p_{i_a}\right)}{p_{i_s}}.
\end{equation}
Using the inequalities (\ref{eqn-3}) and (\ref{eqn-4}), we get
\begin{align*}
\left\vert E_{m_{i_s,i_{a+t}}}\right\vert & = \phi\left(\frac{n}{p_{i_s}p_{i_{a+t}}}\right)\\
 & = \phi\left(\frac{p_{i_1}^{n_{i_1}}\ldots p_{i_{a}}^{n_{i_{a}}}}{p_{i_s}}\right) \phi\left(\frac{p_{i_{a+1}}^{n_{i_{a+1}}}\ldots p_{i_{a+b}}^{n_{i_{a+b}}}p_r^{n_r}}{p_{i_{a+t}}}\right)\\
 & > p_{1}^{n_{1}-1}\ldots p_{r}^{n_{r}-1} \times \frac{\phi\left(p_{i_1}p_{i_2}\ldots p_{i_a}\right)}{p_{i_s}}\times \phi\left(p_{i_{a+1}}\ldots p_{i_{a+b}}\right).
\end{align*}
So
\begin{equation}\label{eqn-5}
\underset{s=1}{\overset{a}{\sum}} \left\vert E_{m_{i_s,i_{a+t}}}\right\vert > p_{1}^{n_{1}-1}\ldots p_{r}^{n_{r}-1} \times \left(\underset{s=1}{\overset{a}{\sum}} \frac{\phi\left(p_{i_1}p_{i_2}\ldots p_{i_a}\right)}{p_{i_s}} \right)\times  \phi\left(p_{i_{a+1}}\ldots p_{i_{a+b}}\right).
\end{equation}
Observe that the right hand side of (\ref{eqn-5}) is independent of $t$. Since $b\geq 2$ and the sets $E_{m_{i_s,i_{a+t}}}$ are pairwise disjoint, we get
\begin{align*}
\left\vert \underset{t=1}{\overset{b}{\bigcup}}\left( \underset{s=1}{\overset{a}{\bigcup}} E_{m_{i_s,i_{a+t}}}\right) \right\vert &=
\underset{t=1}{\overset{b}{\sum}}\left( \underset{s=1}{\overset{a}{\sum}} \left\vert E_{m_{i_s,i_{a+t}}}\right\vert \right)=
b\times \left( \underset{s=1}{\overset{a}{\sum}} \left\vert E_{m_{i_s,i_{a+t}}}\right\vert \right)\\
 & > 2 p_{1}^{n_{1}-1}\ldots p_{r}^{n_{r}-1} \left(\underset{s=1}{\overset{a}{\sum}} \frac{\phi\left(p_{i_1}p_{i_2}\ldots p_{i_a}\right)}{p_{i_s}} \right) \phi\left(p_{i_{a+1}}\ldots p_{i_{a+b}}\right)\\
 & =  p_{1}^{n_{1}-1}\ldots p_{r}^{n_{r}-1} \left(\underset{s=1}{\overset{a}{\sum}} \frac{2\phi\left(p_{i_1}p_{i_2}\ldots p_{i_a}\right)}{p_{i_s}} \right) \phi\left(p_{i_{a+1}}\ldots p_{i_{a+b}}\right)\\
 & > p_{1}^{n_{1}-1}\ldots p_{r}^{n_{r}-1} \left(\underset{s=1}{\overset{a}{\sum}} \frac{p_{i_1}p_{i_2}\ldots p_{i_a}}{p_{i_s}} \right) \phi\left(p_{i_{a+1}}\ldots p_{i_{a+b}}\right).
\end{align*}
The hypothesis that $2\phi(p_1\ldots p_{r-1})>p_1\ldots p_{r-1}$ implies $2\phi(p_{j_1}\ldots p_{j_{l}})>p_{j_1}\ldots p_{j_l}$ for any subset $\{j_1,\ldots, j_l\}$ of $\{1,2,\ldots, r-1\}$. So the last inequality holds in the above. Note that each of the sets $E_{m_{i_s,i_{a+t}}}$ is disjoint from $\underset{j=1}{\overset{b}{\bigcup}} S_{m_{r,i_{a+j}}}$. So
$$\left\vert R\right\vert - \left\vert T\right\vert = \left\vert\underset{j=1}{\overset{b}{\bigcup}} S_{m_{r,i_{a+j}}}\right\vert + \left\vert \underset{t=1}{\overset{b}{\bigcup}}\left( \underset{s=1}{\overset{a}{\bigcup}} E_{m_{i_s,i_{a+t}}}\right) \right\vert - \left\vert T\right\vert.$$
Then
\begin{align*}
\frac{\left\vert R\right\vert - \left\vert T\right\vert}{p_1^{n_1 -1}\ldots p_{r}^{n_{r-1}}}& >  \phi\left( p_{1}p_2\ldots p_{r-1}\right)- p_{i_1}\ldots p_{i_a}\phi(p_{i_{a+1}}\ldots p_{i_{a+b}}) \\
& \;\;\;\;\;\;\;\;\;\;\;\;\;\;\;\;\;\;\; + \phi\left(p_{i_{a+1}}\ldots p_{i_{a+b}}\right) \left(\underset{s=1}{\overset{a}{\sum}} \frac{p_{i_1}p_{i_2}\ldots p_{i_a}}{p_{i_s}} \right)\\
& = \phi\left(p_{i_{a+1}}\ldots p_{i_{a+b}}\right) \left[\phi\left(p_{i_{1}}\ldots p_{i_{a}}\right)- p_{i_1}\ldots p_{i_a}+\left(\underset{s=1}{\overset{a}{\sum}} \frac{p_{i_1}\ldots p_{i_a}}{p_{i_s}} \right)\right]\\
& \geq 0.
\end{align*}
The last inequality holds by Lemma \ref{simple-4}. It follows that $\left\vert R\right\vert > \left\vert T\right\vert$, a contradiction. This proves Claim-2.
\end{proof}

The following proves Theorem \ref{main}(i). Recall the integers $\alpha_{k}$ and $\beta_{k,i_1, i_2,\ldots, i_l}$ defined in Section \ref{upperbounds} for $0\leq k\leq n_r-1$ and subsets $\{i_1,i_2,\ldots, i_l\}$ of $\{1,2,\ldots, r-1\}$.

\begin{proposition}\label{E-m-r-in-A}
The set $E_{m_r}$ is disjoint from $X$. As a consequence, $X=Z(r, n_r-1)$ and
$$\kappa(\mathcal{P}(C_n))=\phi(n) + p_1^{n_1 -1}\ldots p_{r-1}^{n_{r-1} -1} p_r^{n_r -1} \left[p_1p_2\ldots p_{r-1} - \phi(p_1p_2\ldots p_{r-1})\right].$$
\end{proposition}

\begin{proof}
By Proposition \ref{all}, we may assume that all the sets $E_{m_i}$, $1\leq i\leq r-1$, are contained in $B$. We show that $E_{m_r}$ is contained in $A$.

Note that the order of an element in $A$ is of the form $\alpha_j=p_1^{n_1}p_2^{n_2} \ldots p_{r-1}^{n_{r-1}}p_r^j$ for some $j$ with $0\leq j\leq n_r -1$. This follows, since the elements of $A$ are not adjacent with the elements of $E_{m_i}$, $1\leq i\leq r-1$, in $B$.
Let $t\in \{0,1,\ldots, n_r -1\}$ be the largest integer for which $A$ has an element of order
$\alpha_t$. Then Lemma \ref{simple-3} implies that $E_{\alpha_t}\subseteq A$. We claim that $t=n_r -1$.

If $n_r =1$, then there is nothing to prove. So consider $n_r \geq 2$. Suppose that $t< n_r -1$. Since $A\cup B$ is a separation of $\mathcal{P}(\overline{X})$, the sets $E_{\alpha_k}$ ($t+1\leq k\leq n_r -1$) and the subgroups $S_{\beta_{t,l}}$ ($1\leq l\leq r-1$) are contained in $T$. Set
$$P=\underset{k=t+1}{\overset{n_r -1} \bigcup} E_{\alpha_k}\mbox{ and } Q=\underset{l=1}{\overset{r -1} \bigcup} S_{\beta_{t,l}}.$$
Then $|P|+|Q|=|P\cup Q|\leq |T|$. We now calculate $|P|$ and $|Q|$. Since the sets $E_{\alpha_k}$ are pairwise disjoint, we have
\begin{align*}
|P| & = \left\vert E_{\alpha_{t+1}}\right\vert+\left\vert E_{\alpha_{t+2}}\right\vert+\ldots + \left\vert E_{\alpha_{n_r -1}}\right\vert\\
& = \phi\left(p_1^{n_1}p_2^{n_2} \ldots p_{r-1}^{n_{r-1}}\right)\underset{k=t+1}{\overset{n_r -1} \sum} \phi\left(p_r^k\right)\\
& = p_1^{n_1 -1}p_2^{n_2 -1} \ldots p_{r-1}^{n_{r-1}-1}\phi(p_1 p_2\ldots p_{r-1})\left(p_r^{n_r -1}-p_r^t\right).
\end{align*}
Applying a similar calculation as in the proof of Proposition \ref{size-Z}, we get
\begin{align*}
\left\vert Q\right\vert & = p_1^{n_1 -1}p_2^{n_2 -1} \ldots p_{r-1}^{n_{r-1}-1}p_r^t\left[p_1p_2\ldots p_{r-1}-\phi(p_1 p_2\ldots p_{r-1})\right].
\end{align*}
Then
\begin{align*}
\frac{|P|+|Q|-|T|}{p_1^{n_1 -1} \ldots p_{r-1}^{n_{r-1}-1}} & \geq \phi(p_1 \ldots p_{r-1})\left(p_r^{n_r -1}-p_r^t\right)\\
 & \;\;\;\;\;\;\;\;\; -\left[p_1\ldots p_{r-1}-\phi(p_1\ldots p_{r-1})\right]\left(p_r^{n_r -1}-p_r^t\right)\\
 & =\left[2\phi(p_1\ldots p_{r-1}) - p_1\ldots p_{r-1}\right] \left(p_r^{n_r -1}-p_r^t\right) > 0.
\end{align*}
The last inequality holds, since $0\leq t< n_r -1$ and $2\phi(p_1\ldots p_{r-1}) > p_1\ldots p_{r-1}$. It follows that $|P|+|Q|>|T|$, a contradiction. Hence $t=n_r -1$ and $E_{m_r}$ is contained in $A$.

We now show that $X=Z(r, n_r-1)$. Since $E_{m_r}$ is contained in $A$ and $E_{m_i}$, $1\leq i\leq r-1$, are contained in $B$, the subgroups
$$S_{m_{1,r}}, S_{m_{2,r}},\ldots, S_{m_{r-1,r}}$$
of $C_n$ must be contained in $X$. Since $E_n\subseteq X$, it follows that $X$ contains $Z(r, n_r-1)$. Then minimality of $|X|$ implies that $X=Z(r, n_r-1)$ and so
$$\kappa(\mathcal{P}(C_n))=|Z(r, n_r-1)|=\phi(n) + p_1^{n_1 -1}\ldots p_{r-1}^{n_{r-1} -1} p_r^{n_r -1} \left[p_1p_2\ldots p_{r-1} - \phi(p_1p_2\ldots p_{r-1})\right].$$
This completes the proof.
\end{proof}

As a consequence of Theorem \ref{main}(i) and (iii), we now prove Corollary \ref{coro}.

\begin{proof}[{\bf Proof of Corollary\ref{coro}}]
Since $p_1\geq r\geq 2$, we get
\begin{equation*}
\begin{aligned}
\frac{\phi(p_1p_2\ldots p_{r-1})}{p_1p_2\ldots p_{r-1}} & = \left(1-\frac{1}{p_1}\right)\left(1-\frac{1}{p_2}\right)\ldots \left(1-\frac{1}{p_{r-1}}\right)\\
& \geq \left(1-\frac{1}{r}\right)\left(1-\frac{1}{r+1}\right)\ldots \left(1-\frac{1}{2(r-1)}\right)= \frac{1}{2},
\end{aligned}
\end{equation*}
where the inequality is strict except when $r=2$ and $p_1=2$. Thus $2 \phi(p_1\ldots p_{r-1}) \geq p_1\ldots p_{r-1}$, with equality if and only if $(r,p_1)= (2,2)$. Then the corollary follows from Theorem \ref{main}(i) and (iii).
\end{proof}

\section{Proof of Theorem \ref{main-1}}\label{last-sec}

Here $r=3$. Since $2\phi(p_1p_2)< p_1p_2$, it follows from the proof Corollary \ref{coro} that $p_1 <3=r$. So $p_1=2$ and $n=2^{n_1}p_2^{n_2}p_3^{n_3}$. Since $2< p_{2}< p_{3}$, we have $p_2\geq 3$ and $p_3\geq 5$.
Let $X$ be a subset of $C_{n}$ of minimum size such that $\mathcal{P}(\overline{X})$ is disconnected. Then, using the bound (\ref{bound-2}), we have
$|X|\leq \phi(n)+2^{n_{1}-1}p_{2}^{n_{2}-1}\left[p_{3}^{n_{3}-1}\phi (p_{2})+2\right].$ Setting $\Gamma = X\setminus E_{n}$, we get
\begin{equation}\label{gamma-size}
|\Gamma|\leq 2^{n_{1}-1}p_{2}^{n_{2}-1}\left[p_{3}^{n_{3}-1}\phi (p_{2})+2\right].
\end{equation}

\begin{proposition}
$E_{m_1}$ is disjoint from $X$.
\end{proposition}

\begin{proof}
Otherwise, $E_{m_1}\subseteq \Gamma$ by Lemma \ref{simple-3}. Since the identity element of $C_n$ is in $\Gamma$ but not in $E_{m_1}$, we get $|E_{m_1}|< |\Gamma|$. On the other hand, using (\ref{gamma-size}), we have
\begin{equation*}
|E_{m_1}|-|\Gamma|\geq
\begin{cases}
p_{2}^{n_{2}-1}\left[p_{3}^{n_{3}-1}(p_{2}-1)(p_{3}-2)-2\right], &  \text{if $n_{1}=1$}\\
2^{n_{1}-2}p_{2}^{n_{2}-1}\left[p_{3}^{n_{3}-1}(p_{2}-1)(p_{3}-3)-4\right], & \text{if $n_{1}\geq 2$}
\end{cases}.
\end{equation*}
Since $p_{2} \geq 3$ and $p_{3 }\geq 5$, it follows that $|E_{m_1}|-|\Gamma|\geq 0$, a contradiction.
\end{proof}

Fix a separation $A\cup B$ of $\mathcal{P}(\overline{X})$. By the above proposition, $E_{m_1}$ is contained either in $A$ or in $B$. Without loss of generality, we may assume that $E_{m_1}\subseteq A$.

\begin{proposition}\label{at-least}
At least one of $E_{m_2}$ and $E_{m_3}$ is disjoint from $X$.
\end{proposition}

\begin{proof}
Otherwise, both $E_{m_2}$ and $E_{m_3}$ are contained in $\Gamma$ by Lemma \ref{simple-3} and so $|E_{m_2}| +|E_{m_3}|\leq |\Gamma|$. Set $\theta=|E_{m_2}| +|E_{m_3}|- |\Gamma|$. Using (\ref{gamma-size}), the following can be verified:
\begin{equation*}
\theta\geq
\begin{cases}
2^{n_{1}-1}\left( p_{3}-3\right), &  \text{if $n_{2}=1=n_{3}$}\\
2^{n_{1}-1}\left(p_{3}^{n_{3}-2}[p_{3}(p_{3}-p_{2})+\phi (p_{2}p_{3})]-2 \right), & \text{if $n_{2}=1$ and $n_{3}\geq 2$}\\
2^{n_{1}-1}p_{2}^{n_{2}-2}\left(\phi (p_{2}p_{3}) -2p_2  \right), & \text{if $n_{2}\geq 2$ and $n_{3} =1$}\\
2^{n_{1}-1}p_{2}^{n_{2}-2}\left( p_{3}^{n_{3}-2}\phi (p_{2})\left(p_{3}^{2}-p_{3}-p_{2}\right)-2p_{2} \right), & \text{if $n_{2}\geq 2$ and $n_{3}\geq 2$}\\
\end{cases}.
\end{equation*}
Since $p_{2} \geq 3$ and $p_{3 }\geq 5$, it follows that $\theta >0$ in all the four cases. This gives $|E_{m_2}| +|E_{m_3}| > |\Gamma|$, a contradiction.
\end{proof}

\begin{proposition}\label{r=3-e-m-2}
$E_{m_2}$ is disjoint from $X$.
\end{proposition}

\begin{proof}
Otherwise, $E_{m_2}$ is contained in $\Gamma$ by Lemma \ref{simple-3}. Then, by Proposition \ref{at-least}, $E_{m_3}$ is disjoint from $X$ and so is contained either in $A$ or in $B$.\\

{\bf Case 1:} $E_{m_3}\subseteq B$. Since $E_{m_1}\subseteq A$, the subgroup $S_{m_{1,3}}$ must be contained in $\Gamma$ and so $|E_{m_2}| +|S_{m_{1,3}}| \leq |\Gamma|$. We calculate $|E_{m_2}| +|S_{m_{1,3}}| - |\Gamma|$ using (\ref{gamma-size}). If $n_{2}=1$, then
$$|E_{m_2}| +|S_{m_{1,3}}| - |\Gamma|\geq  2^{n_{1}-1} \left[p_{3}^{n_{3}}- 2\right] > 0.$$
If $n_{2}\geq 2$, then
$$ \frac{|E_{m_2}| +|S_{m_{1,3}}| - |\Gamma|}{2^{n_{1}-1}p_{2}^{n_{2}-2}} \geq  p_{3}^{n_{3}-1}\left[p_{2}^{2} + \phi (p_{2})(p_{3}-p_{2}-1)\right]-2p_{2} > 0.$$
Thus $|E_{m_2}| +|S_{m_{1,3}}| > |\Gamma|$, a contradiction.\\

{\bf Case 2:} $E_{m_3}\subseteq A$. Since $E_{m_2}\subseteq X$ and $A\cup B$ is a separation of $\mathcal{P}(\overline{X})$, the order of any element of $B$ is of the form $2^{n_{1}}p_{2}^{k}p_{3}^{n_{3}},$
where $0\leq k\leq n_{2}-2$ (this is possible only when $n_2\geq 2$). Let $t\in\{0,1,\ldots, n_2-2\}$ be the largest integer for which $B$ has an element of order $2^{n_{1}}p_{2}^{t}p_{3}^{n_{3}}$. Then the sets  $$E_{2^{n_{1}}p_{2}^{t+1}p_{3}^{n_{3}}},E_{2^{n_{1}}p_{2}^{t+2}p_{3}^{n_{3}}},\ldots, E_{2^{n_{1}}p_{2}^{n_{2}-2}p_{3}^{n _{3}}},E_{2^{n_{1}}p_{2}^{n_{2}-1}p_{3}^{n _{3}}}=E_{m_2}$$
and the two subgroups
$$S_{2^{n_{1}}p_{2}^{t}p_{3}^{n_{3}-1}},\; S_{2^{n_{1}-1}p_{2}^{t}p_{3}^{n_{3}}}$$
are contained in $\Gamma$. Set $P=\underset{k=t+1}{\overset{n_{2}-1}{\bigcup}} E_{2^{n_{1}}p_{2}^{k}p_{3}^{n_{3}}}$ and $Q=S_{2^{n_{1}}p_{2}^{t}p_{3}^{n_{3}-1}}\cup S_{2^{n_{1}-1}p_{2}^{t}p_{3}^{n_{3}}}$ . Then $|P|+|Q|=|P\cup Q|\leq |\Gamma|$. We have $|Q|=2^{n_{1}-1}p_{2}^{t}p_{3}^{n_{3}-1}(p_{3}+1)$ and
$$|P| = \sum _{k=t+1} ^{n_{2}-1} \phi (2^{n_{1}}p_{2}^{k}p_{3}^{n_{3}})= 2^{n_{1}-1}p_{3}^{n_{3}-1}(p_{3}-1)\left(p_{2}^{n_{2}-1}-p_{2}^{t}\right).$$
So
\begin{eqnarray*}
 |P|+|Q|& = & 2^{n_{1}-1}p_{3}^{n_{3}-1}\left[(p_{3}-1)(p_{2}^{n_{2}-1}-p_{2}^{t})+p_{2}^{t}(p_{3}+1)\right]\\
& = & 2^{n_{1}-1}p_{2}^{n_{2}-1}p_{3}^{n_{3}-1}(p_{3}-1)+ 2^{n_{1}}p_{2}^{t}p_{3}^{n_{3}-1}.
\end{eqnarray*}
Then $|P|+|Q|- |\Gamma|\geq 2^{n_{1}-1} p_{2}^{n_{2}-1}[p_{3}^{n_{3}-1}(p_{3}-p_{2})-2]+2^{n_{1}}p_{2}^{t}p_{3}^{n_{3}-1} >  0$ for any $t$. This gives $|P|+|Q|> |\Gamma|$, a contradiction.
\end{proof}

\begin{proposition}\label{r=3-e-m-2-1}
$E_{m_2}$ is contained in $A$.
\end{proposition}

\begin{proof}
Proposition \ref{r=3-e-m-2} implies that $E_{m_2}$ is contained either in $A$ or in $B$. Suppose that $E_{m_2}\subseteq B$. Since $E_{m_1}\subseteq A$, the subgroup $S_{m_{1,2}}$ is contained in $\Gamma$. Consider the set $E_{m_3}$, which would be contained either in $A, B$ or $\Gamma$. We show that none of these possibilities occurs. If $E_{m_3}\subseteq \Gamma$, then
\begin{equation*}
\frac{|E_{m_3}|+|S_{m_{1,2}}| - |\Gamma|}{2^{n_{1}-1}p_{2}^{n_{2}-1}}\geq
\begin{cases}
p_{3}-2, & \text{if  $n_3=1$}\\
p_{3}^{n_{3}-2}\left(p_{3}^2-p_{2}+1\right)-2, & \text{if  $n_3\geq 2$}\\
\end{cases}.
\end{equation*}
This gives $|E_{m_3}|+|S_{m_{1,2}}| > |\Gamma|$, a contradiction. If $E_{m_3}\subseteq A$, then the subgroup $S_{m_{2,3}}$ is contained in $\Gamma$. Since
$$|S_{m_{1,2}}\cup S_{m_{2,3}}|- |\Gamma|\geq 2^{n_{1}-1}p_{2}^{n_{2}-1}\left[ p_{3}^{n_{3}-1}(p_{3}-p_{2}+2)-2\right]>0,$$
we get $|S_{m_{1,2}}\cup S_{m_{2,3}}| > |\Gamma|$, a contradiction. Finally, assume that $E_{m_3}\subseteq B$. Then the subgroup $S_{m_{1,3}}$ is contained in $\Gamma$.
In this case, we get
$$|S_{m_{1,2}}\cup S_{m_{1,3}}|- |\Gamma|\geq 2^{n_{1}-1}p_{2}^{n_{2}-1}\left[ p_{3}^{n_{3}}-2\right]>0,$$
giving $|S_{m_{1,2}}\cup S_{m_{1,3}}| > |\Gamma|$, a contradiction.
\end{proof}

\begin{proposition}\label{r=3-order}
The order of any element of $B$ is $2^{n_{1}}p_{2}^{n_{2}}$.
\end{proposition}

\begin{proof}
We have $E_{m_1}\subseteq A$ by our assumption and $E_{m_2}\subseteq A$ by Proposition \ref{r=3-e-m-2-1}. Since $A\cup B$ is a separation of $\mathcal{P}(\overline{X})$, the order of any element of $B$ is of the form $2^{n_{1}}p_{2}^{n_{2}}p_{3}^{k},$
where $0\leq k\leq n_{3}-1$. Let $t\in\{0,1,\ldots,n_3 -1\}$ be the largest integer for which $B$ has an element of order $2^{n_{1}}p_{2}^{n_{2}}p_{3}^{t}$ (and so $E_{2^{n_{1}}p_{2}^{n_{2}}p_{3}^{t}}\subseteq B$). Then the sets
$$E_{2^{n_{1}}p_{2}^{n_{2}}p_{3}^{t+1}},E_{2^{n_{1}}p_{2}^{n_{2}}p_{3}^{t+2}},\ldots ,E_{2^{n_{1}}p_{2}^{n_{2}}p_{3}^{n _{3}-1}}$$
and the two subgroups
$$S_{2^{n_{1}}p_{2}^{n_{2}-1}p_{3}^{t}},\; S_{2^{n_{1}-1}p_{2}^{n_{2}}p_{3}^{t}}$$
are contained in $\Gamma$. Set $P_1=\underset{k=t+1}{\overset{n_{3}-1}{\bigcup}} E_{2^{n_{1}}p_{2}^{n_{2}}p_{3}^{k}}$ and $Q_1=S_{2^{n_{1}}p_{2}^{n_{2}-1}p_{3}^{t}}\bigcup S_{2^{n_{1}-1}p_{2}^{n_{2}}p_{3}^{t}}$. Then $|P_1|+|Q_1|=|P_1\cup Q_1|\leq |\Gamma|$. On the other hand, we have
$$|Q_1|=2^{n_{1}-1}p_{2}^{n_2 -1}p_{3}^{t}(p_{2}+1)\mbox{ and } |P_1| = 2^{n_{1}-1}p_{2}^{n_{2}-1}(p_{2}-1)\left(p_{3}^{n_{3}-1}-p_{3}^{t}\right). $$
So $|P_1|+|Q_1|  =  2^{n_{1}-1}p_{2}^{n_{2}-1}p_{3}^{n_{3}-1}(p_{2}-1)+ 2^{n_{1}}p_{2}^{n_2 -1}p_{3}^{t}.$
Then, using (\ref{gamma-size}), we get
$$|P_1|+|Q_1|- |\Gamma|\geq 2^{n_{1}}p_{2}^{n_{2}-1}(p_{3}^{t}-1).$$
If $t\geq 1$, then it would follow that $|P_1|+|Q_1|> |\Gamma|$ which is not possible. So $t=0$ and every element in $B$ is of order $2^{n_{1}}p_{2}^{n_{2}}$.
\end{proof}

Now, Proposition \ref{r=3-order} together with the facts that $E_{m_1}$, $E_{m_2}$ are contained in $A$ imply the sets $E_{2^{n_{1}}p_{2}^{n_{2}}p_{3}^{k}}$ with $k\in\{1,2,\ldots, n_3 -1\}$ and the two subgroups $S_{2^{n_{1}}p_{2}^{n_{2}-1}}, S_{2^{n_{1}-1}p_{2}^{n_{2}}}$ are contained in $\Gamma$. Also $E_n\subseteq X$. Since $Z(3,0)$ is precisely the union of these sets, it follows that $X$ contains $Z(3,0)$. Then, by the minimality of $|X|$, we get $X=Z(3,0)$ and hence
$$\kappa(\mathcal{P}(C_n))=|X|=|Z(3,0)|=\phi(n)+2^{n_{1}-1}p_{2}^{n_{2}-1}\left( p_3^{n_3 -1}\phi(p_2) +2 \right).$$
This completes the proof of Theorem \ref{main-1}.

\vskip .5cm

\noindent{\bf Address}:\\
School of Mathematical Sciences\\
National Institute of Science Education and Research, Bhubaneswar (HBNI)\\
At/Po - Jatni, District - Khurda, Odisha - 752050, India.\\
E-mails: sriparna@niser.ac.in, klpatra@niser.ac.in, bksahoo@niser.ac.in
\end{document}